\documentclass{amsart}
\usepackage{amsmath,amssymb}

\newtheorem{remark}{Remark}

\newtheorem{definition}{Definition}

\newtheorem{proposition}{Proposition}

\newtheorem{conjecture}{Conjecture}

\usepackage{epsfig}
\usepackage{url}
\usepackage{amssymb}
\usepackage{amsmath}
\usepackage{amsfonts}

\def\Dbar{\leavevmode\lower.6ex\hbox to 0pt{\hskip-.23ex \accent"16\hss}D}

\def\bR{{\mbox{\bf R}}}
\def\bZ{{\mbox{\bf Z}}}

\def\paf{{\mbox{\rm PAF}}}

\begin{document}

\title{Three new lengths for cyclic Legendre pairs} 

\author {N. A. Balonin}
\address{Saint-Petersburg State University of Aerospace 
Instrumentation, Saint-Petersburg, Russian Federation}
\email{korbendfs@mail.ru}
\author {D. {\v{Z}}. {\Dbar}okovi{\'c}}
\address{University of Waterloo, 
Department of Pure Mathematics and Institute for Quantum Computing,
Waterloo, Ontario, N2L 3G1, Canada}
\email{djokovic@uwaterloo.ca}

%% \thanks{\NSERC}

\date{}

\begin{abstract}
There are 20 odd integers $v$ less than 200 for which the existence of Legendre pairs of length $v$ is undecided. 
The smallest among them is $v=77$. We have constructed 
Legendre pairs of lengths 91, 93 and 123 reducing the number 
of undecided cases to 17.
\end{abstract}

\maketitle

\section{Introduction}

It is conjectured that the cyclic Legendre pairs of odd lengths $>1$ always exist. See the next section for the definition of 
the Legendre pairs and Legendre difference families (DF). It is known that they exist for odd lengths $v$ in the range $2<v<76$. 
The smallest unresolved case is $v=77$. 
According to \cite{FGS:2001}, there are four series of known cyclic Legendre pairs of odd length $v>1$ (the first three are infinite):

(i) $v$ is a prime number;

(ii) $2v+1$ is a power of a prime number;

(iii) $v+1$ is a power of 2;

(iv) $v=pq$, where $p$ and $q$ are prime numbers and $q-p=2$.

We refer to (i) as the {\em classical series} because the 
construction is based on the sequence of the classical Legendre symbols, see \cite{FGS:2001}. 
The case (ii) is the {\em Szekeres series} provided by the well known series of so called Szekeres difference sets (in fact they are difference families), see \cite{Szekeres:1969,JS:2017}.
The series (iii) is known as the {\em Galois series} 
\cite{Schroeder:1984} and (iv) is the {\em twin-prime series}, 
see e.g. \cite[Theorem 9.4]{MP:2013}. 

The series (iii) and (iv) as well as (i) for $v\equiv 3 \pmod{4}$ are obtained from the three well known series of difference sets having the  parameters $(v;(v-1)/2;(v-3)/4)$. We refer to such Legendre pairs as type 1 (see sec. \ref{sec:type}).

If we start with the list of odd integers $v$ in the range 
$76<v<200$ and remove those which satisfy at least one of the conditions (i)-(iv) above we obtain the list of 20 integers: 
\begin{eqnarray*}
&& 77, 85, 87, 91, 93, 115, 117, 123, 129, 133, 145, \\
&& 147, 159, 161, 169, 175, 177, 185, 187, 195.
\end{eqnarray*}
This is in fact the list of all cases with $v<200$ for which the question of existence of cyclic Legendre pairs is unresolved.

In the paper \cite[p. 80]{FGS:2001} the authors list 22 odd lengths $<200$ for which they assert that the existence question of cyclic Legendre pairs is unresolved. However, the lengths 121 and 171 should not have been included in that list since 
$2\cdot 121+1=243=3^5$ and $2\cdot 171+1=343=7^3$ are prime 
powers. (On the other hand according to 
\cite[sec. 4]{CKKP:2007} the number 57 should have been included.)

There are two other series of Legendre DFs in elementary abelian 
groups which include some cyclic cases. One of them appears in 
\cite{JSW:1973} and the other in \cite[Theorem 3.1]{Ding:JCD:2008}. However, while they provide new cyclic Legendre DFs they do not give new lengths $v$ in the cyclic case.

Our main result is in sec. \ref{sec:New} where we give the first examples of cyclic Legendre pairs of lengths 91, 93 and 123. Thereby we reduce to 17 the number of the undecided cases listed above. 

According to \cite{FGS:2001}, exhaustive computer searches for cyclic Legendre pairs where carried out for all odd $v<48$. In 
sec. \ref{sec:more} we consider the odd integers $v$ in the range $48<v<76$ and we list the new cyclic Legendre DFs of type 2 that we constructed. Only for $v=69$ and $v=75$ we failed to find any new pairs.

\section{Notation and definitions}

Let $G$ be a finite abelian group (written additively) and let 
$v$ denote its order. For any function $f:G\to\bR$ its periodic auto-correlation function, $\paf_f:G\to\bR$, is defined by the formula $\paf_f(s)=\sum_{x\in G} f(x)f(x+s)$. 
We refer to $s$ as the {\em shift variable}. 

\begin{definition} We say that an ordered pair of functions 
$(f,g)$ mapping $G\to\{+1,-1\}$ is a {\em Legendre pair} on 
$G$ if $\paf_f(s)+\paf_g(s)=-2$ for all nonzero shifts $s$.
(For $s=0$ we have $\paf_f(0)=\paf_g(0)=v$.)
\end{definition}

For any function $f:G\to\{+1,-1\}$ we set 
$G_f=\{x\in G: f(x)=-1\}$. 

\begin{proposition} 
An ordered pair of functions $(f,g):G\to\{+1,-1\}$ is a Legendre 
pair on $G$ if and only if $(G_f,G_g)$ is a difference family 
in $G$ with parameters $(v;k_1,k_2;\lambda)$ where 
$k_1=|G_f|$, $k_2=|G_g|$ and $\lambda=k_1+k_2-(v+1)/2$. 
In particular the existence of Legendre pairs on $G$ implies 
that $v$ must be odd. 
\end{proposition} 

\begin{proof} This follows immediately from Theorems 3 and 4 of 
\cite{DK:CompMethods:2019}.
\end{proof}

\begin{remark}
It is customary to require that the length $v$ of a Legendre pair is $>1$. However, according to the above definition, if $G$ is a trivial group then any pair of functions $G\to\{+1,-1\}$ is
a Legendre pair. The condition in the definition holds by 
default (there are no nonzero shifts). 
\end{remark}

If we introduce the additional parameter $n=k_1+k_2-\lambda$ then we have $v=2n-1$. By using the well known equation 
$k_1(k_1-1)+k_2(k_2-1)=\lambda (v-1)$ one can easily show that 
$$
\{k_1,k_2\} \subseteq \{ \frac{v-1}{2}, \frac{v+1}{2} \}.
$$

In view of the above proposition, we shall refer to the difference families $(X,Y)$ in $G$ having the parameters 
$(v;k_1,k_2;\lambda)$ with $v=2n-1$ as {\em Legendre DFs}.

We say that a Legendre pair on a group $G$ is {\em cyclic} if 
the group $G$ is cyclic. In this note we deal only with the 
cyclic Legendre pairs and we may assume that $G=\bZ_v$, the 
additive group of integers modulo $v$.

We give a simple example to introduce the notation that we will 
use in the rest of this note.

{\em Example:} $v=39$. In this example $v=13\cdot 3$ and so 
$\bZ_v^*$ (the group of units of the ring $\bZ_v$) is isomorphic 
to $\bZ_{12}\times\bZ_2$. Then $H=\{1,16,22\}$ is the unique subgroup of $\bZ_v^*$ of order 3. 
There exists an $H$-invariant Legendre DF with parameter set
$(39;19,19;18)$, namely
\begin{eqnarray*}
\quad \\
X &=& H\{0,1,2,3,4,12,14\} \\ 
Y &=& H\{0,2,3,4,8,14,19\} \\
\end{eqnarray*}
In general, if $H$ is a subgroup of $\bZ_v^*$ and $S$ a subset
of $\bZ_v$ then the product $HS$ is defined to be 
$HS=\{hs \pmod{v}: h\in H,~ s\in S\}$. Note that 
$H\{0\}=\{0\}$.

\section{Cyclic Legendre pairs of new lengths $v=91,93,123$} 
\label{sec:New}

We have constructed four pairwise nonequivalent Legendre DFs 
$(X_i,Y_i)$ of length $91$. For the definition of equivalence see sec. \ref{sec:equiv}. Only one DF is constructed 
for each of the lengths $93$ and $123$. Instead of Legendre pairs, we list the corresponding difference families. 
In each case, each block is a union of orbits of a fixed subgroup ($H$, $H_1$ or $H_2$) of order 3 or 5 of $\bZ_v^*$.

\centerline{$v=91$}

Four pairwise nonequivalent Legendre DFs:
\begin{eqnarray*}               
&& (91;45,45;44) \quad H_1=\{1,16,74\},~ H_2=\{1,9,81\} \\
X_1 &=& H_1\{1,2,7,14,15,17,19,22,25,28,38,43,44,50,55\} \\
Y_1 &=& H_1\{2,3,10,11,14,17,20,22,28,43,44,45,49,50,55\} \\
X_2 &=& H_1\{1,4,5,8,9,11,15,22,27,28,34,38,43,49,50\} \\
Y_2 &=& H_1\{8,9,10,11,14,17,22,25,28,33,34,38,44,50,55\} \\
X_3 &=& H_1\{2,3,5,9,10,14,15,20,27,28,33,34,38,50,55\} \\
Y_3 &=& H_1\{3,4,11,14,19,25,27,28,33,34,43,44,45,50,55\} \\
X_4 &=& H_2\{2,5,14,16,19,20,23,24,29,30,37,40,46,48,49\} \\
Y_4 &=& H_2\{2,4,6,8,13,14,16,23,30,37,38,39,40,46,49\} \\
\end{eqnarray*}

\centerline{$v=93$}

Only one Legendre DF:
\begin{eqnarray*}
&& (93;46,46;45) \quad  H=\{1,25,67\} \\
X &=& H\{0,1,2,3,5,8,10,12,13,16,22,24,43,44,47,48\} \\
Y &=& H\{0,1,3,4,5,9,11,12,18,20,22,37,40,43,44,51\} \\
\end{eqnarray*}

\centerline{$v=123$}

Only one Legendre DF:
\begin{eqnarray*}
&& (123;61,61;60) \quad H=\{1,10,16,37,100\} \\
X &=& H\{0,1,3,6,11,13,28,29,33,35,43,45,59\} \\
Y &=& H\{4,5,6,11,14,15,18,19,22,28,33,41,45\} \\
\end{eqnarray*}

\section{Type 1 and type 2} \label{sec:type}

Let $(X,Y)$ be a Legendre DF in $\bZ_v$. It is easy to see that if $X$ is a difference set then $Y$ must be a difference set too. In that case we say that $(X,Y)$ (and its corresponding 
Legendre pair) is of {\em type $1$}, and otherwise that it is of 
{\em type $2$}. The Legendre pairs in the Galois and the 
twin-prime series as well as those in the classical series with 
$v$ a prime number $\equiv 3 \pmod{4}$ are of type 1. Note that two equivalent Legendre DFs must have the same type. Hence the study of type 1 Legendre DFs essentially reduces to the study of difference sets. For that reason we shall consider only the  
Legendre DFs of type 2.

As mentioned earlier, it is conjectured that cyclic Legendre DFs exist for all odd lengths $v>2$. We propose a bit stronger 
version.

\begin{conjecture} \label{cj:type2}
Legendre DFs of type $2$ exist for all odd lengths $v>8$.
\end{conjecture}

One of the objectives of this note is to verify this conjecture 
for $v<76$. It follows from \cite{FGS:2001} that the conjecture 
is true for $v<48$. If $v$ is a prime number $\equiv 1 \pmod{4}$ then the classical Legendre pair of length $v$ is of type 2. One 
can verify that the Legendre pairs in the Szekeres series having 
length $v$ in the interval $4<v<76$ are of type 2. Thus, in 
order to verify the above conjecture for $v<76$ it suffices 
to verify it in the cases $v=49,55,57,59,67,71$. 
This will be done in the next section.

We do not know whether all Legendre pairs of length $v>4$ in the 
Szekeres series are of type 2.

\section{New Legendre DFs of type 2} \label{sec:more}

In this section we list the cyclic Legendre DFs of type 2 and length $v>48$ that we have constructed. We imposed the restriction $v>48$ because for $v<48$ exhaustive searches 
have been carried out \cite{FGS:2001}.

\centerline{$v=49$} 

The Legendre DF below is not equivalent to the one in \cite[p. 85]{FGS:2001}.
\begin{eqnarray*}
&& (49;24,24;23) \quad H=\{1,18,30\} \\
X &=& H\{1,2,8,9,13,24,26,29\} \\
Y &=& H\{2,3,4,6,7,8,12,37\} \\
\end{eqnarray*}

\centerline{$v=51$} 

The two Legendre DFs below together with the one in 
\cite[p. 85]{FGS:2001} and another one from the Szekeres series are pairwise nonequivalent.
\begin{eqnarray*}
&& (51;25,25;24) \quad H=\{1,16\} \\
X_1 &=& H\{0,1,2,4,6,8,19,24,25,28,35,38,41\} \\
Y_1 &=& H\{0,2,4,5,9,14,15,18,21,22,25,31,35\} \\
X_2 &=& H\{1,2,9,11,17,18,19,21,24,25,28,38,41\} \\
Y_2 &=& H\{1,3,4,5,8,9,15,17,18,19,21,22,31\} \\
\end{eqnarray*}

\centerline{$v=53$} 

All ten Legendre DFs listed below are pairwise nonequivalent.
The first five are known: the first belongs to the classical series, the second is from \cite{JSW:1973}, the third from \cite{Ding:JCD:2008}, the fourth from the Szekeres series, and the fifth from \cite[p. 85]{FGS:2001}. We have constructed many Legendre DFs for $v=53$ but we recorded only five of them (the last five in the list below). 

\begin{eqnarray*}
&& (53;26,26;25) \quad 
H=\{1,10,13,15,16,24,28,36,42,44,46,47,49\} \\
X_1 &=& H\{1,4\}, \quad Y_1 = H\{2,5\} \\
X_2 &=& H\{1,2\}, \quad Y_2 = H\{1,5\} \\
X_3 &=& H\{1,2\}, \quad Y_3 = H\{2,5\} \\
X_4 &=& \{1,4,5,6,8,14,16,17,19,21,22,23,26,28,29,33,35,38, \\
&& \quad\quad  40,41,42,43,44,46,50,51\} \\
Y_4 &=& \{4,8,9,10,12,14,15,18,19,21,22,23,24,29,30,31,32, \\
&& \quad\quad  34,35,38,39,41,43,44,45,49\} \\
X_5 &=& \{5,7,12,13,15,18,19,24,26,28,30,33,35,36,37,38,39, \\
&& \quad\quad   42,43,44,46,47,48,50,51,52\} \\
Y_5 &=& \{4,7,8,10,11,14,15,20,21,23,24,25,26,29,30,32,37, \\
&& \quad\quad  40,42,44,47,48,49,50,51,52\} \\
X_6 &=& \{0,1,2,3,5,9,10,11,12,14,17,24,25,26,28,29,34,35,40, \\
&&        44,45,46,47,48,50,51\} \\
Y_6 &=& \{0,2,4,6,7,8,12,14,17,19,20,21,22,24,27,28,30,31,34, \\
&&        35,40,44,46,48,49,52\} \\
X_7 &=& \{0,2,3,4,7,10,11,12,13,16,17,18,21,23,32,33,37,38, \\
&&        39,40,41,42,45,49,50,52\} \\
Y_7 &=& \{0,1,3,6,7,12,13,15,20,21,23,24,25,31,33,35,37,38, \\
&&        40,42,44,46,47,48,50,51\} \\
X_8 &=& \{0,1,2,6,7,9,12,13,14,15,18,19,20,24,34,36,37,38,39, \\
&&        41,43,45,46,49,51,52\} \\
Y_8 &=& \{0,1,5,8,9,11,12,15,16,20,21,23,24,25,31,33,35,38, \\
&&        40,41,43,44,47,49,51,52\} \\
X_9 &=& \{0,6,8,9,10,12,14,15,23,26,28,30,31,32,33,36,37,38, \\
&&        40,41,42,43,48,49,50,52\} \\
Y_9 &=& \{0,1,2,3,6,8,10,11,14,15,16,18,20,23,29,31,34,35,38, \\
&&        39,41,42,45,47,48,52\} \\
X_{10} &=& \{0,1,6,7,8,13,14,15,17,18,19,21,22,23,24,26,28, \\
&&          32,37,38,40,46,48,49,50,51\} \\
Y_{10} &=& \{0,2,6,7,10,13,14,15,16,18,19,21,22,25,26,30,31, \\
&&          32,33,34,36,40,43,46,48,50\} \\
\end{eqnarray*}

\centerline{$v=55$} 

The Legendre DF below is not equivalent to the one listed in \cite[p. 85]{FGS:2001}.
\begin{eqnarray*}
&& (55;27,27;26) \quad H=\{1,34\} \\
X &=& H\{1,2,6,7,8,9,10,11,15,16,21,24,27,37,50\} \\
Y &=& H\{1,2,3,8,16,17,19,20,21,25,27,29,37,40,42\} \\
\end{eqnarray*}

\centerline{$v=57$} 

In this case only two nonequivalent Legendre DFs are known.
The first one was constructed in 2007 \cite{CKKP:2007} and the 
second one constructed very recently in \cite[Section 2.4]{ABH:2020}. We have constructed the six Legendre DFs below. The first five of them, together with the two known DFs, are pairwise nonequivalent. The sixth is equivalent to the one constructed in \cite{ABH:2020}. 
\begin{eqnarray*}               
&& (57;28,28;27) \quad H=\{1,7,49\} \\
X_1 &=& H\{0,2,3,4,8,16,23,24,30,31\} \\
Y_1 &=& H\{0,2,3,4,6,8,16,23,24,29\} \\
X_2 &=& H\{0,2,8,10,12,23,24,29,30,31\} \\
Y_2 &=& H\{0,3,4,5,6,22,23,24,29,31\} \\
X_3 &=& H\{0,1,3,5,6,10,16,23,29,30\} \\
Y_3 &=& H\{0,1,2,3,15,16,22,24,29,31\} \\
X_4 &=& H\{0,1,2,4,6,11,15,29,30,31\} \\
Y_4 &=& H\{0,8,11,12,22,23,24,29,30,31\} \\
X_5 &=& H\{0,2,3,4,6,10,15,16,29,31\} \\
Y_5 &=& H\{0,1,3,8,10,15,16,23,24,29\} \\
X_6 &=& H\{0,2,3,4,5,11,15,16,22,30\} \\
Y_6 &=& H\{0,1,2,4,15,16,22,23,24,30\} \\
\end{eqnarray*}

\centerline{$v=59$} 

First examples of Legendre DFs of length 59 and type 2:
\begin{eqnarray*}
&& (59;29,29;28) \\
X_1 &=& \{0,1,2,3,4,5,6,10,12,13,15,16,19,20,21,25,27,30,31,33,
37,38,39,41, \\
&& 43,44,45,52,56\} \\
Y_1 &=& \{0,1,3,4,5,6,7,10,12,13,14,15,17,18,20,23,26,27,28,30,
34,35,36,39, \\
&& 43,45,48,50,55\} \\
X_2 &=& \{0,1,2,3,4,5,7,8,9,10,12,15,16,17,22,24,25,26,28,29,33,
34,38,39,42, \\
&& 44,48,50,53\} \\
Y_2 &=& \{0,2,3,4,6,7,9,10,12,14,15,18,19,21,23,25,29,30,31,32,
33,36,38,39, \\
&& 43,46,49,50,51\} \\
X_3 &=& \{0,1,2,3,4,5,7,8,9,10,14,16,19,20,21,24,27,28,30,32,
36,37,38,41, \\
&& 45,47,48,51,54\} \\
Y_3 &=& \{0,2,3,4,5,6,8,9,10,12,14,16,17,19,24,25,26,28,29,31,
32,34,39, \\
&& 42,43,44,50,54,55\} \\
X_4 &=& \{0,1,2,3,4,6,9,11,12,13,15,16,18,19,20,23,24,29,30,32,
36,37,38,40, \\
&& 41,46,47,49,56\} \\
Y_4 &=& \{0,1,2,4,6,7,9,10,11,13,14,15,17,21,22,25,26,28,30,31,
33,35,36,41, \\
&& 43,45,46,47,53\} \\
X_5 &=& \{0,1,2,3,4,6,10,12,13,14,16,18,19,21,23,26,27,28,29,31,
32,36,37,40, \\
&& 42,43,44,49,51\} \\
Y_5 &=& \{0,1,2,4,5,7,8,10,11,13,14,17,18,20,21,22,25,26,31,32,
36,37,38,40, \\
&& 42,44,45,47,52\} \\
X_6 &=& \{0,1,2,3,5,6,7,8,9,11,13,15,17,18,20,24,27,28,29,31,32,
38,40,41,44, \\
&& 45,46,54,55\} \\
Y_6 &=& \{0,1,2,3,6,7,9,10,12,13,15,17,19,20,21,24,25,27,31,32,
33,36,40,41, \\
&& 43,44,46,49,51\} \\
\end{eqnarray*}

\centerline{$v=61$} 

In this case apart from the classical Legendre DF there is another one provided by a lemma of J. Seberry Wallis \cite{JSW:1973}, see also \cite[Lemma 2]{Djokovic:Facta:1997}.
The Legendre DF below is not equivalent to any of them.
\begin{eqnarray*}
&& (61;30,30;29) \quad H=\{1,9,20,34,58\} \\
X &=& H\{2,3,4,5,12,26\} \\
Y &=& H\{3,4,5,10,12,13\} \\
\end{eqnarray*}

\centerline{$v=63$}

The six new Legendre DFs below and the one from the Szekeres  series are all pairwise nonequivalent.
\begin{eqnarray*}
&& (63;31,31;30) \quad H_1=\{1,4,16\},~ H_2=\{1,25,58\},~ 
                       H_3=\{1,8,11,23,25,58\} \\
X_1 &=& H_1\{2,3,6,10,11,22,23,30,31,42,47\} \\
Y_1 &=& H_1\{1,2,7,9,10,11,14,15,21,31,47\} \\
X_2 &=& H_1\{1,3,9,13,14,15,21,22,23,30,47\} \\
Y_2 &=& H_1\{3,5,7,9,10,11,15,21,22,23,30\} \\
X_3 &=& H_1\{1,2,3,6,9,11,14,21,22,30,31\} \\
Y_3 &=& H_1\{1,2,3,6,7,9,11,15,22,42,47\} \\
X_4 &=& H_1\{2,3,7,9,10,14,15,21,26,30,43\} \\
Y_4 &=& H_1\{3,6,7,10,11,14,21,27,30,43,47\} \\
X_5 &=& H_2\{1,3,6,7,15,17,20,27,29,40,42\} \\
Y_5 &=& H_2\{3,5,7,8,10,15,17,21,27,30,40\} \\
X_6 &=& H_3\{0,2,9,10,15,19,27\} \\
Y_6 &=& H_3\{0,2,5,7,9,15,27\} \\
\end{eqnarray*}

\centerline{$v=65$}

The Legendre DF below is not equivalent to the one in the   Szekeres series.
\begin{eqnarray*}
&& (65;32,32;31) \quad H=\{1,16,61\} \\
X &=& H\{1,5,6,9,18,20,22,23,24,26,35,52\} \\
Y &=& H\{0,1,3,7,11,13,19,22,23,24,36,50\} \\
\end{eqnarray*}

\centerline{$v=67$}

The following Legendre DF gives the first example of Legendre pairs of length 67 and type 2:
\begin{eqnarray*}
&& (67;33,33;32) \quad H=\{1,29,37\} \\
X &=& H\{1,3,5,6,10,16,17,30,34,41,53\} \\
Y &=& H\{2,4,6,9,12,15,16,18,25,32,41\} \\
\end{eqnarray*}

\centerline{$v=71$}

We give the first example of a Legendre DF of length 71 and of 
type 2:
\begin{eqnarray*}
&& (71;35,35;34) \quad H=\{1,5,25,54,57\} \\
X &=& H\{1,2,3,6,11,14,27\} \\
Y &=& H\{1,2,3,9,14,18,42\} \\
\end{eqnarray*}

\centerline{$v=73$}

The Legendre DF below is not equivalent to the one in the classical series.
\begin{eqnarray*}
&& (73;36,36;35) \quad H=\{1,8,64\} \\
X &=& H\{2,5,6,7,9,11,12,17,18,26,35,42\} \\
Y &=& H\{1,2,3,7,9,13,18,21,26,33,35,42\} \\
\end{eqnarray*}

\centerline{$v=111$}

The Legendre DF below is not equivalent to the one belonging to  the Szekeres series.
\begin{eqnarray*}
&& (111;55,55;54) \quad H=\{1,10,100\} \\
X &=& H\{0,1,2,3,4,7,8,9,13,16,21,22,27,41,42,44,54,62,63\} \\
Y &=& H\{0,1,3,4,5,6,7,8,11,16,17,21,26,27,52,53,55,63,64\} \\
\end{eqnarray*}

\centerline{$v=121$}

Note that $2v+1=243=3^5$ is a prime power $\equiv 3 \pmod{4}$. 
We list below two Legendre DFs $(X_i,Y_i)$, $i=1,2$. The first 
one is equivalent to the DF in the Szekeres series. 
The block $X_1$ is skew and $Y_1$ is symmetric. 
We have constructed the second Legendre DF $(X_2,Y_2)$ with 
$X_2=X_1$ and verified that the two DFs are nonequivalent. Although $Y_2$ is not symmetric, the second pair still qualifies as a Szekeres difference set according to 
\cite[Definition 5.6]{JS:2017}) 
\begin{eqnarray*}
&& (121;60,60;59) \quad H=\{1,3,9,27,81\} \\
X_1 &=& H\{4,10,11,20,25,26,34,35,38,40,67,76\} \\
Y_1 &=& H\{1,7,8,10,16,20,26,31,35,38,61,94\} \\
X_2 &=& X_1 \\
Y_2 &=& H\{1,4,5,8,11,13,17,20,22,26,34,76\} \\
\end{eqnarray*}

\section{Equivalence of Legendre pairs} \label{sec:equiv}

To define the equivalence, we need first to define the 
elementary transformations on the set of Legendre pairs 
on a given finite abelian group $G$ of odd order $v$. 
(We assume that $G$ is written additively.) If $f$ is a 
function $G\to \{+1,-1\}$ and $s\in G$ then we say that 
the function $G\to \{+1,-1\}$ sending $x\to f(s+x)$ is 
the {\em translate} of $f$ by $s$.

The {\em elementary transformations} of a Legendre pair 
$(f,g)$ are the following:

$(i)$ interchange $f$ and $g$;

$(ii)$ replace $f$ by $-f$;

$(iii)$ replace $f$ by its translate by $s\in G$;

$(iv)$ replace $f$ by $f\circ\iota$, where $\iota$ is the 
automorphism of $G$ sending each $x\in G$ to its inverse 
$-x$;

$(v)$
 replace $(f,g)$ by $(f\circ\alpha,g\circ\alpha)$ where 
$\alpha$ is an automorphism of $G$.

\begin{definition}
We say that two Legendre pairs on $G$ are {\em equivalent} 
if one can be transformed to the other by performing a finite 
sequence of elementary transformations.
\end{definition}

The effect on $(G_f,G_g)$ of the above elementary transformations is as follows:

$(i)'$ interchange $G_f$ and $G_g$;

$(ii)'$ replace $G_f$ by $G\setminus G_f$;

$(iii)'$ replace $G_f$ by the translate $G_f -s$;

$(iv)'$ replace $G_f$ by $-G_f$;

$(v)'$ replace $(G_f,G_g)$ by 
$(\alpha^{-1}(G_f),\alpha^{-1}(G_g))$.

We define the equivalence of Legendre DFs on $G$ by using the 
$(i)'-(v)'$ as elementary transformations of pairs $(G_f,G_g)$. 
Then two Legendre pairs are equivalent if and only if 
their Legendre DFs are equivalent. We remark that because of $(ii)'$, two equivalent Legendre DFs may have different parameter sets.

\section{Acknowledgements} \label{sec:acknowledgements}

The research of the first author leading to these results has received financial support of the Ministry of Science and Higher 
Education of the Russian Federation, agreement 
No. FSRF-2020-0004. The research of the second author was enabled in part by support provided by SHARCNET (http://www.sharcnet.ca) and Compute Canada (http://www.computecanada.ca).

\end{document}